\documentclass[
a4paper,
11pt, 
DIV=12,
toc=bibliography, 
headsepline, 
parskip=half-,
abstract=true,
]{scrartcl}

\usepackage[english]{babel}
\usepackage[utf8]{inputenc}
\usepackage[T1]{fontenc}
\usepackage[english,cleanlook]{isodate}
\usepackage{lmodern}
\usepackage[babel=true]{microtype}
\usepackage{amsmath}
\usepackage{mathtools}
\usepackage{amssymb}
\usepackage{amsthm}
\usepackage{mathrsfs}
\usepackage{scrlayer-scrpage} 
\usepackage{enumitem}
\usepackage{tikz} 
\usepackage{tikz-3dplot} 
\usetikzlibrary{plotmarks,positioning} 

\ihead{M.\,B.\,Schulz}
\ohead{\csname @title\endcsname}
\setkomafont{pageheadfoot}{\footnotesize}

\providecommand{\R}{\mathbb{R}}
\providecommand{\Z}{\mathbb{Z}}
\providecommand{\N}{\mathbb{N}}
\providecommand{\B}{\mathbb{B}} 
\providecommand{\dih}{\mathbb{D}} 
\providecommand{\hsd}{\mathscr{H}} 
\providecommand{\symdiff}{\mathbin{\triangle}} 
\providecommand{\apar}{\alpha}  
\providecommand{\solC}{\mathcal{C}} 
\providecommand{\solS}{\mathcal{S}} 
\providecommand{\solD}{\mathcal{D}} 
\providecommand{\solU}{\mathcal{U}}
\providecommand{\solA}{\mathcal{A}} 
\providecommand{\gsum}{\mathfrak{g}}  
\providecommand{\bsum}{\mathfrak{b}} 
\providecommand{\rotation}{\textsf{R}} 

\DeclarePairedDelimiter\abs{\lvert}{\rvert} 
\DeclarePairedDelimiter\interval{]}{[}
\DeclarePairedDelimiter\Interval{[}{[}
\DeclarePairedDelimiter\intervaL{]}{]} 

\pgfmathsetmacro{\unitscale}{0.95} 
\providecommand{\picpfad}{figures-ellipsoid}
\tikzset{axis/.style={black,densely dotted}, 
/ellipsoid/.cd,
apar/.initial=1, 
bpar/.initial=1, 
cpar/.initial=1, 
theta/.initial=66, 
phi/.initial=135, 
scale/.initial=\unitscale, 
side/.initial=0, 
}
\newenvironment{ellipsoid}[1][]{%
\tikzset{/ellipsoid/.cd,#1}%
\pgfmathsetmacro{\apar}{\pgfkeysvalueof{/ellipsoid/apar}}%
\pgfmathsetmacro{\bpar}{\pgfkeysvalueof{/ellipsoid/bpar}}%
\pgfmathsetmacro{\cpar}{\pgfkeysvalueof{/ellipsoid/cpar}}%
\pgfmathsetmacro{\thetapar}{\pgfkeysvalueof{/ellipsoid/theta}}%
\pgfmathsetmacro{\phipar}{\pgfkeysvalueof{/ellipsoid/phi}}%
\pgfmathsetmacro{\xscale}{abs(\apar*cos(\phipar)+(\bpar*sin(\phipar))^2/(\apar*cos(\phipar)))/sqrt((\bpar*tan(\phipar)/\apar)^2+1)}%
\pgfmathsetmacro{\thetaO}{atan(\xscale*\cpar*tan(\thetapar)/(\apar*\bpar))}%
\pgfmathsetmacro{\yscale}{sqrt(abs(\cpar*sin(\thetapar))^2+abs(\apar*\bpar*cos(\thetapar)/\xscale)^2)}%
\pgfmathsetmacro{\yslant}{sign(tan(\phipar))*sign(\bpar-\apar)*sqrt(abs(
(\apar*cos(\phipar))^2
+(\apar*sin(\phipar)*cos(\thetapar))^2
+(\bpar*sin(\phipar))^2
+(\bpar*cos(\phipar)*cos(\thetapar))^2
-\xscale^2-(\yscale*cos(\thetaO))^2 
))/\yscale}%
\pgfmathsetmacro{\phiO}{acos((\apar/\xscale)*cos(\phipar))}%
\tdplotsetmaincoords{\thetapar}{\phipar}%
\begin{tikzpicture}[tdplot_main_coords,line cap=round,line join=round,scale=\pgfkeysvalueof{/ellipsoid/scale},semithick,baseline={(0,0,0)}]
\ifnum\pgfkeysvalueof{/ellipsoid/side}<1
\tdplotsetmaincoords{\thetaO}{\phiO}
\begin{scope}[tdplot_main_coords,yscale=\yscale,xscale=\xscale,yslant=\yslant,black!30]
\path[tdplot_screen_coords,inner color=white,outer color=cyan!5!black!20](0,0)circle(1);
\foreach\laengengrad in {0,45,...,135}{
\tdplotsetthetaplanecoords{\laengengrad}
\pgfmathsetmacro{\start}{-atan(1/(tan(\thetaO)*(cos(\laengengrad)*sin(\phiO)-cos(\phiO)*sin(\laengengrad))))}
\ifdim\start pt>0pt \pgfmathsetmacro{\start}{\start-180} \fi
\tdplotdrawarc[tdplot_rotated_coords]{(0,0,0)}{1}{180+\start}{360+\start}{}{}
}%
\foreach\breitengrad in {-60,-30,...,60}{
\pgfmathsetmacro{\visiblerange}{180-acos(-sign(\breitengrad)*min(1,abs(tan(\breitengrad)*tan(90-\thetaO))))}%
\ifdim\visiblerange pt>0pt
\tdplotdrawarc{(0,0,{sin(\breitengrad)})}{{cos(\breitengrad)}}{\phiO+90-\visiblerange}{\phiO+90+\visiblerange}{}{}\fi}%
\end{scope}\fi%
}{\ifnum\pgfkeysvalueof{/ellipsoid/side}>-1
\tdplotsetmaincoords{\thetaO}{\phiO}
\begin{scope}[tdplot_main_coords,yscale=\yscale,xscale=\xscale,yslant=\yslant,black]
\foreach\laengengrad in {0,45,...,135}{
\tdplotsetthetaplanecoords{\laengengrad}
\pgfmathsetmacro{\start}{-atan(1/(tan(\thetaO)*(cos(\laengengrad)*sin(\phiO)-cos(\phiO)*sin(\laengengrad))))}
\ifdim\start pt>0pt \pgfmathsetmacro{\start}{\start-180} \fi
\tdplotdrawarc[tdplot_rotated_coords]{(0,0,0)}{1}{\start}{180+\start}{}{}
}%
\foreach\breitengrad in {-60,-30,...,60}{
\pgfmathsetmacro{\visiblerange}{acos(-sign(\breitengrad)*min(1,abs(tan(\breitengrad)*tan(90-\thetaO))))}%
\ifdim\visiblerange pt>0pt
\tdplotdrawarc{(0,0,{sin(\breitengrad)})}{{cos(\breitengrad)}}{\phiO-90+\visiblerange}{\phiO-90-\visiblerange}{}{}\fi}%
\draw[tdplot_screen_coords](0,0)circle(1);
 \end{scope}\fi%
\begin{scope}[tdplot_screen_coords]
\path(current bounding box.south west)coordinate(SW);
\path(current bounding box.north east)++(0,-0.5)coordinate(NE);
\pgfresetboundingbox
\useasboundingbox(SW)rectangle(NE);
\end{scope}
\end{tikzpicture}%
}

\usepackage[initials]{amsrefs}
 
\usepackage[
pdftitle={Equivariant free boundary minimal discs and annuli in ellipsoids},
pdfauthor={Mario B. Schulz}, 
pdfborder={0 0 0},
]{hyperref}

\newtheoremstyle{plain}
  {\topsep}   
  {0pt}       
  {\itshape}  
  {0pt}       
  {\bfseries} 
  {.}         
  {5pt plus 1pt minus 1pt} 
  {}          
\newtheoremstyle{remark}
  {\topsep}   
  {0pt}       
  {\normalfont}
  {0pt}       
  {\itshape}  
  {.}         
  {5pt plus 1pt minus 1pt} 
  {}          

\theoremstyle{plain}
\newtheorem{theorem}{Theorem}[section] 
\newtheorem{lemma}[theorem]{Lemma}
\newtheorem{corollary}[theorem]{Corollary}  
 
\theoremstyle{remark}
\newtheorem{remark}[theorem]{Remark}

\title{Equivariant free boundary minimal discs and~annuli in ellipsoids} 
\author{Mario B. Schulz}
\date{\vspace*{-3ex}} 
 
\newcommand\printaddress{{
\setlength{\parindent}{17pt}
\bigskip
\par
{\scshape Mario B. Schulz}
\newline 
Università di Trento, 
Dipartimento di Matematica, 
via Sommarive 14, 
38123 Povo, 
Italy
\newline
\textit{E-mail address:} 
\texttt{mario.schulz@unitn.it}
}} 

\begin{document}

\maketitle

\begin{abstract} 
We employ equivariant variational methods to construct new examples of nonplanar free boundary minimal discs in ellipsoids. 
We also prove that every ellipsoid contains at least three distinct embedded free boundary minimal annuli with dihedral symmetry. 
\end{abstract}

\section{Introduction}

The classical Plateau problem asks for a minimal surface bounded by a given Jordan curve. 
This problem was famously resolved by Douglas \cite{Douglas1931} and Radó \cite{Rado1930}, who independently proved the existence of a solution with the topology of a disc. 
A few years later, Courant \cite[Part~II]{Courant1940} addressed the ``Plateau problem with free boundaries'', which concerns the existence of nontrivial area minimising surfaces whose boundaries are free to move on a given manifold. 
Given a compact, three-dimensional ambient manifold $M$ with boundary $\partial M$, 
we call a compact, two-dimensional submanifold of $M$ which is stationary (and not necessarily minimising) for the area functional among all surfaces $\Sigma\subset M$ with boundary $\partial\Sigma=\Sigma\cap\partial M$ a \emph{free boundary minimal surface}.  
Equivalently, a free boundary minimal surface has vanishing mean curvature and meets the ambient boundary $\partial M$ orthogonally along its own boundary. 
 
The case where $M$ is the Euclidean unit ball $\B^3\subset\R^3$ 
has attracted considerable attention, partly due to its intriguing connection with the optimisation problem for the first Steklov eigenvalue on surfaces with boundary \cite{FraserSchoen2011,FraserSchoen2016,KarpukhinKokarevPolterovich2014,GirouardLagace2021}.
Existence results have been obtained using 
gluing methods \cite{FolhaPacardZolotareva2017,KapouleasLi2017,KapouleasMcGrath2020,KapouleasWiygul2017,KapouleasZou2021,CSWnonuniqueness,CSWstackings}, min-max methods \cite{GruterJost1986,Ketover2016FBMS,Li2015,CarlottoFranzSchulz2022,FranzSchulz2023,HaslhoferKetover} and via Steklov eigenvalue optimisation \cite{Petrides,KarpukhinKusnerMcGrathStern}. 
Equally interesting is the question whether a given free boundary minimal surface is unique in a given class of solutions. 
Nitsche \cite{Nitsche1985} proved the uniqueness of the equatorial disc 
in the class of immersed free boundary minimal discs in $\B^3$ up to ambient isometries. 
This result has been generalised to higher codimensions by Fraser and Schoen \cite{FraserSchoen2015}. 
A famous conjecture asserts the uniqueness of the critical catenoid in the class of embedded free boundary minimal annuli in $\B^3$ up to ambient isometries \cite{FraserLi2014}. 
In \cite{McGrath2018,KusnerMcGrath2020}, the uniqueness of the critical catenoid has been proved under additional symmetry assumptions.  
In general however, embedded free boundary minimal surfaces in $\B^3$ are nonunique in the class of solutions with the same topology and symmetry group \cite{CSWnonuniqueness}. 
  
The construction of free boundary minimal surfaces in more general ambient manifolds $M$ has been pioneered by Struwe \cite{Struwe1984}, who constructed parametric minimal discs with free boundary constraint to surfaces $S\subset\R^3$ diffeomorphic to the sphere, and by Grüter and Jost \cite{GruterJost1986} who proved the existence of an embedded free boundary minimal disc in any convex domain $M\subset\R^3$. 
This emphasis on solutions with the topology of a disc is reminiscent of the classical Plateau problem and their existence has been investigated by several authors \cite{Fraser2000,Laurain2019,Lin2020,LiZhou2021}. 
Haslhofer and Ketover \cite{HaslhoferKetover} proved that any strictly convex ball $M$ with nonnegative Ricci-curvature contains at least two embedded free boundary minimal discs. 
Moreover the area of their second solution is strictly less than twice the area of the Grüter--Jost solution.  
An interesting consequence of this result is \cite[Corollary 1.5]{HaslhoferKetover} stating that ellipsoids 
\begin{align}
\label{eqn:ellipsoid}
M_\apar\vcentcolon=\{(x_1,x_2,x_3)\in\R^3\mid (x_1/\apar_1)^2+(x_2/\apar_2)^2+(x_3/\apar_3)^2\leq 1\}	
\end{align}
with $\alpha=(\alpha_1,\alpha_2,\alpha_3)\in\interval{0,\infty}^3$ 
satisfying $\apar_3\geq2\max\{\apar_1,\apar_2\}$ contain a \emph{nonplanar}, embedded free boundary minimal disc, in addition to the three planar solutions 
\begin{align}
\label{eqn:planar}
\solD_\iota&\vcentcolon=M_\apar\cap\{x_\iota=0\}, \qquad \iota\in\{1,2,3\}.
\end{align}
In the case $\apar_1=\apar_2$, this statement recovers an independent result of Petrides \cite{Petrides} about nonplanar, embedded free boundary minimal discs in $M_\apar$ which converge to the planar solution $\solD_3$ with multiplicity $2$ as $\apar_3\to\infty$. 
These discoveries are in stark contrast with Nitsche's \cite{Nitsche1985} aforementioned uniqueness result in $\mathbb{B}^3$ and solve a long-standing open problem originally raised by Smyth \cite[p.\,411]{Smyth1984} in the '80s. (see also \cite[p.\,335]{DHKW1992}). 
Petrides' approach relies on the optimisation of combinations of first and second Steklov eigenvalues on the disc while Haslhofer and Ketover employed a two-parameter min-max construction in arbitrary convex balls with nonnegative Ricci-curvature. 
In this article, we prove several existence results demonstrating that a one-parameter equivariant min-max approach suffices to construct different types of nonplanar free boundary minimal discs and annuli in Euclidean ellipsoids.  

\textbf{Notation.} 
We equip $\R^3$ with standard Cartesian coordinates $x_1,x_2,x_3$ and denote the $k$-dimensional Hausdorff measure on $\R^3$ by $\hsd^k$. 
Given $\iota\in\{1,2,3\}$, the intersection of the ellipsoid $M_\apar$ defined in \eqref{eqn:ellipsoid} with the $x_\iota$-axis is denoted by $\xi_\iota$ and $\rotation_\iota$ denotes the rotation of angle $\pi$ around the $x_\iota$-axis. 
By definition, $\rotation_\iota M_\apar=M_\apar$ for every $\iota\in\{1,2,3\}$ and any choice of $\alpha\in\interval{0,\infty}^3$. 
Let $\dih_1$ be the group of Euclidean isometries generated by $\rotation_1$ and 
let $\dih_2$ be the group of Euclidean isometries generated by $\{\rotation_1,\rotation_2\}$. 
Note that $\dih_1\simeq\Z_2$ and $\rotation_3=\rotation_2\circ\rotation_1\in\dih_2\simeq\Z_2\times\Z_2$.  
Our convention is consistent with the notation in \cite{CarlottoFranzSchulz2022,CSWnonuniqueness,FranzSchulz2023}, where $\dih_n$ denotes the dihedral group of order $2n$ acting on $\R^3$.  
 
\begin{theorem}\label{thm:solC}
For any $\alpha\in\interval{0,\infty}^3$ satisfying $\apar_3\geq2\apar_1$ and $\apar_3>\apar_2$ 
the ellipsoid $M_\apar$ contains an embedded, $\dih_1$-equivariant free boundary minimal disc $\solC$ with the following properties. 
\begin{enumerate}[label={\normalfont(\roman*)},nosep]
\item\label{thm:solC-i}   $\solC$ is nonplanar with area $\hsd^2(\solC)<2\hsd^2(\solD_3)$.  
\item\label{thm:solC-ii}  $\solC$ intersects the segment $\xi_1$ exactly once and the intersection is orthogonal. 
\item\label{thm:solC-iii} $\solC$ has equivariant index equal to $1$. 
\end{enumerate} 
\end{theorem}
 
Theorem~\ref{thm:solC} extends the existence results stated in \cite[Corollary 1.5]{HaslhoferKetover} and \cite{Petrides} by relaxing the assumption on $\apar$. 
It is however possible that $\solC$ coincides with the solutions from \cite{HaslhoferKetover,Petrides} in ellipsoids where those results apply. 
In contrast, the following theorem establishes the existence of a novel type of nonplanar free boundary minimal disc (see Figure \ref{fig:solC}).

\begin{theorem}\label{thm:solS}
For any $\alpha\in\interval{0,\infty}^3$ satisfying $\apar_3\geq3\apar_2$ and $\apar_3>\apar_1$ 
the ellipsoid $M_\apar$ contains an embedded, $\dih_1$-equivariant free boundary minimal disc $\solS$ with the following properties. 
\begin{enumerate}[label={\normalfont(\roman*)},nosep]
\item\label{thm:solS-i}    $\solS$ is nonplanar with area strictly between $\hsd^2(\solD_3)$ and $3\hsd^2(\solD_3)$.  
\item\label{thm:solS-ii}   $\solS$ contains the segment $\xi_1$. 
\item\label{thm:solS-iii}  $\solS$ has equivariant index equal to $1$. 
\end{enumerate} 
\end{theorem}

\begin{figure}%
\centering
\begin{ellipsoid}[apar=3,bpar=2,cpar=6]
\draw(0,0,0)node[scale=\unitscale,inner sep=0]{\includegraphics[width={12cm}]{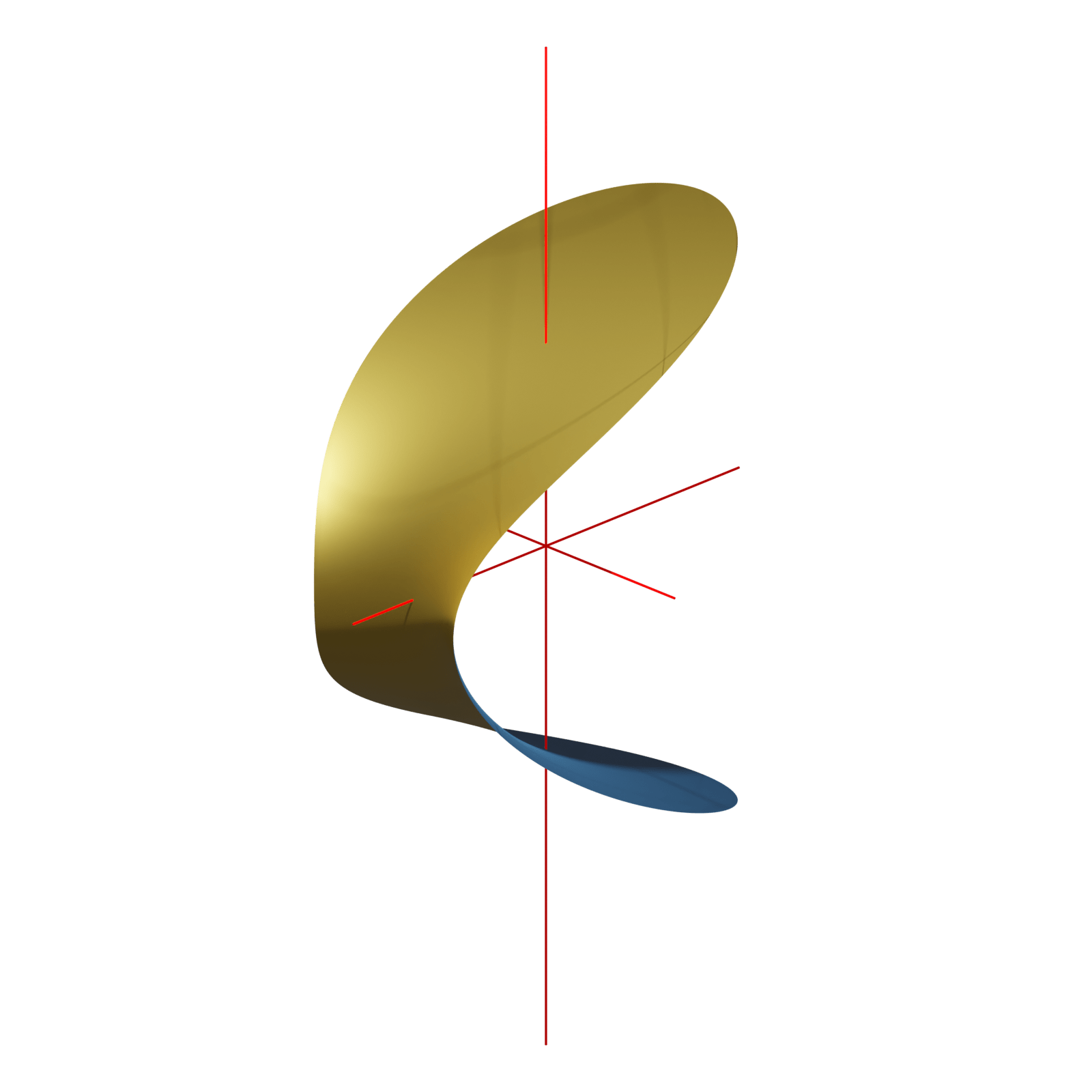}};
\begin{scope}[axis]
\draw[->](\apar,0,0)--(5,0,0)node[below right={1ex and 0},inner sep=0]{$x_1$};
\draw[->](0,\bpar,0)--(0,5,0)node[below left={1ex and 0},inner sep=0]{$x_2$}; 
\end{scope}
\pgfresetboundingbox
\path(5,0,0)(0,5,0);
\end{ellipsoid}
\hfill
\begin{ellipsoid}[apar=3,bpar=2,cpar=6]
\draw(0,0,0)node[scale=\unitscale,inner sep=0]{\includegraphics[width={12cm}]{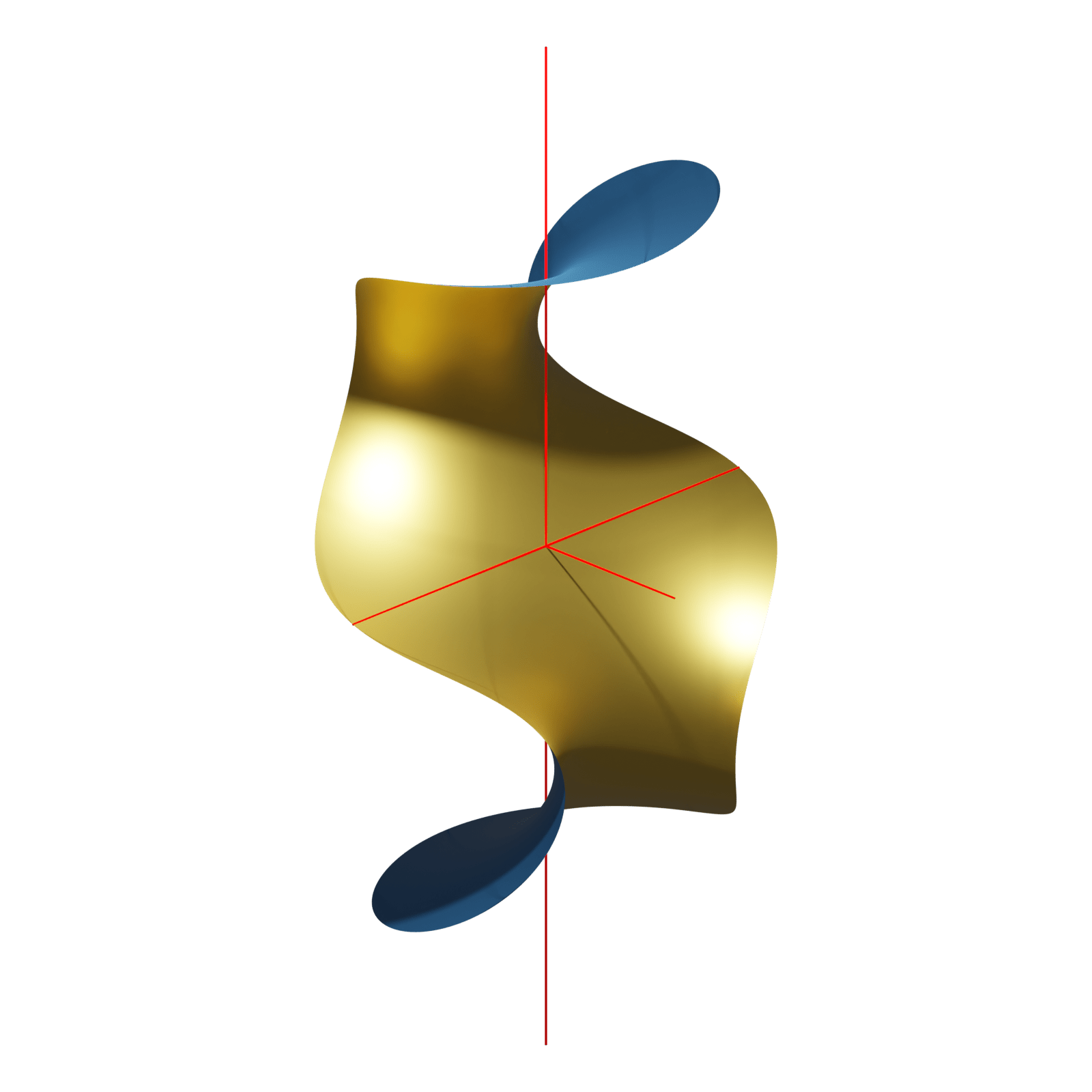}};
\begin{scope}[axis]
\draw[->](\apar,0,0)--(5,0,0)node[below right={1ex and 0},inner sep=0]{$x_1$};
\draw[->](0,\bpar,0)--(0,5,0)node[below left={1ex and 0},inner sep=0]{$x_2$}; 
\path[red!90!black](0,0,0)--(\apar,0,0)node[midway,above,inner sep=2pt]{$\xi_1$};
\end{scope}
\pgfresetboundingbox
\path(5,0,0)(0,5,0);
\end{ellipsoid} 
\caption{$\dih_1$-equivariant free boundary minimal discs $\solC$ and $\solS$ in $M_{\apar}$ for $\apar=(3,2,6)$.
}%
\label{fig:solC}%
\end{figure}

In \cite{CarlottoFranzSchulz2022}, the authors introduced for all $2\leq n\in\N$ an effective $\dih_n$-sweepout of the Euclidean unit ball $\B^3$ in order to construct embedded free boundary minimal surfaces in $\B^3$ with connected boundary and genus equal to $n-1$. 
The case $n=1$ would have been vacuous in that setting, because a solution in $\B^3$ with connected boundary and genus zero is necessarily a flat disc \cite{Nitsche1985}. 
The punchline of this article is that we can prove Theorem~\ref{thm:solS} by 
mimicking the sweepout construction in \cite[§\,2]{CarlottoFranzSchulz2022} for $n=1$ (i.\,e.~genus $g=0$) in $M_\apar$, where the resulting min-max free boundary minimal surface is not necessarily a flat disc.  
Similarly, Theorem~\ref{thm:solC} can be proved by emulating the sweepout from \cite[§\,5]{FranzSchulz2023} for $n=1$ in $M_\apar$. 
We outline these constructions in section~\ref{sec:sweepout}, establish the required min-max width estimate in section~\ref{sec:width} and prove the main theorems in section~\ref{sec:control}.  
In general, the ellipsoid $M_\apar$ is not equivariant with respect to any dihedral group $\dih_n$ with $n>2$ but for $n=2$, it is natural to expect that a $\dih_2$-equivariant min-max approach could establish the existence of free boundary minimal surfaces in $M_\apar$ with nontrivial topology. 
We investigate this idea in Theorem \ref{thm:annulus} below.  

Maximo, Nunes and Smith \cite[Theorem~1.1]{MaximoNunesSmith2017} showed that any compact, strictly convex domain $K\subset\R^3$ contains at least one embedded free boundary minimal annulus.   
In the case where the domain $K$ is an arbitrary ellipsoid, we improve this result by proving the  existence of at least three embedded, $\dih_2$-equivariant free boundary minimal annuli:  

\begin{theorem}\label{thm:annulus}
For any $\alpha\in\interval{0,\infty}^3$ the ellipsoid $M_\apar$ contains three embedded, $\dih_2$-equivariant free boundary minimal annuli $\solA_1,\solA_2,\solA_3$ with the following properties for every $\iota\in\{1,2,3\}$. 
\begin{enumerate}[label={\normalfont(\roman*)},nosep]
\item\label{thm:annulus-i}    $\solA_\iota$ has area $\hsd^2(\solA_\iota)<2\min_{\ell\in\{1,2,3\}}\hsd^2(\solD_\ell)$.  
\item\label{thm:annulus-ii}   $\solA_\iota$ is disjoint from the segment $\xi_\iota$ and intersects the other two segments orthogonally.  
\item\label{thm:annulus-iii}  $\solA_\iota$ has equivariant index equal to $1$. 
\end{enumerate} 
\end{theorem} 

\begin{figure}%
\begin{ellipsoid}[apar=2,bpar=4,cpar=5]
\draw(0,0,0)node[scale=\unitscale]{\includegraphics[width={12cm}]{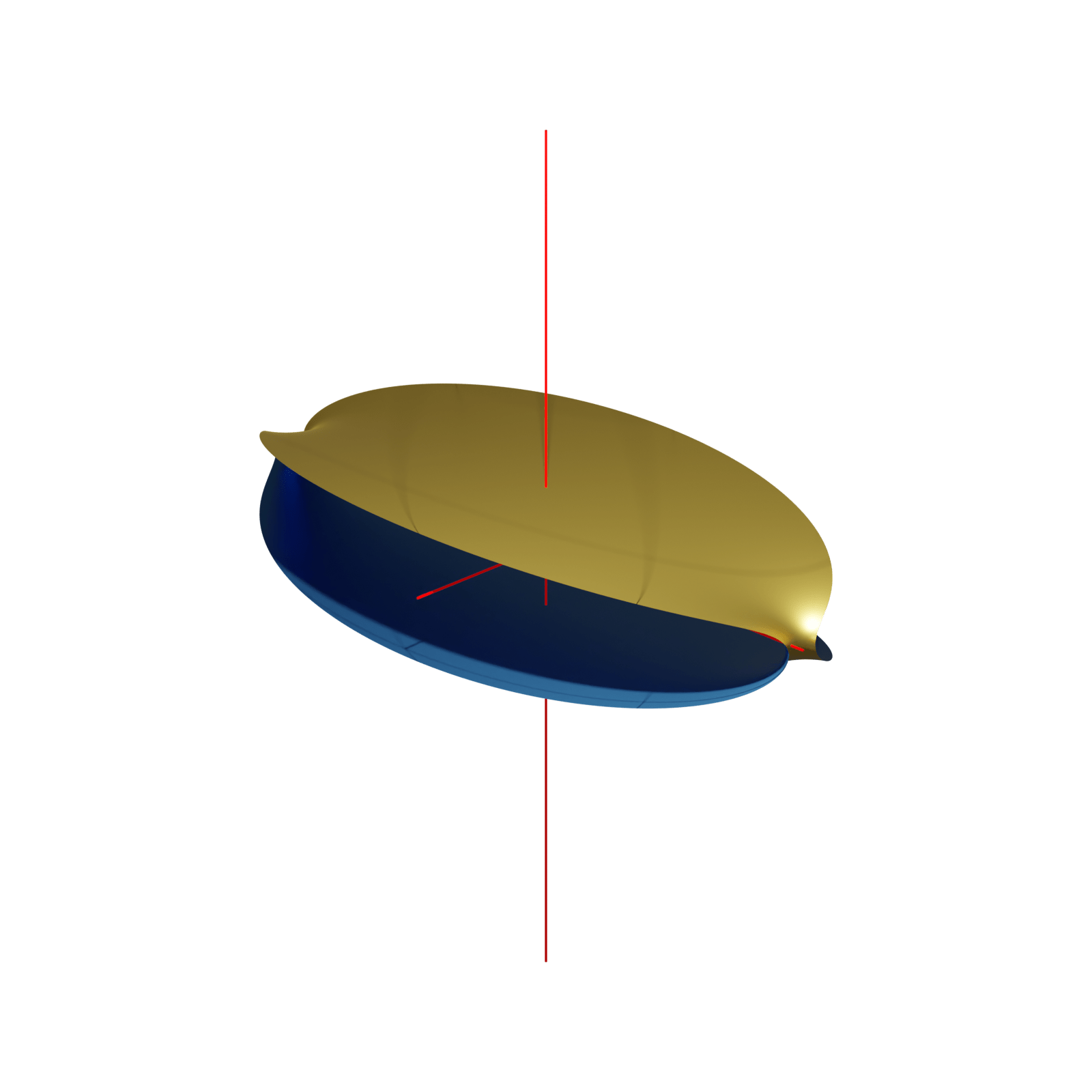}};
\begin{scope}[axis]
\draw[->](\apar,0,0)--(5,0,0)node[below right={1ex and 0},inner sep=0]{$x_1$};
\draw[->](0,\bpar,0)--(0,5,0)node[below]{$x_2$};
\end{scope}
\pgfresetboundingbox
\path(5,0,0)(0,5,0);
\end{ellipsoid}
\hfill
\begin{ellipsoid}[apar=2,bpar=4,cpar=5]
\draw(0,0,0)node[scale=\unitscale]{\includegraphics[width={12cm}]{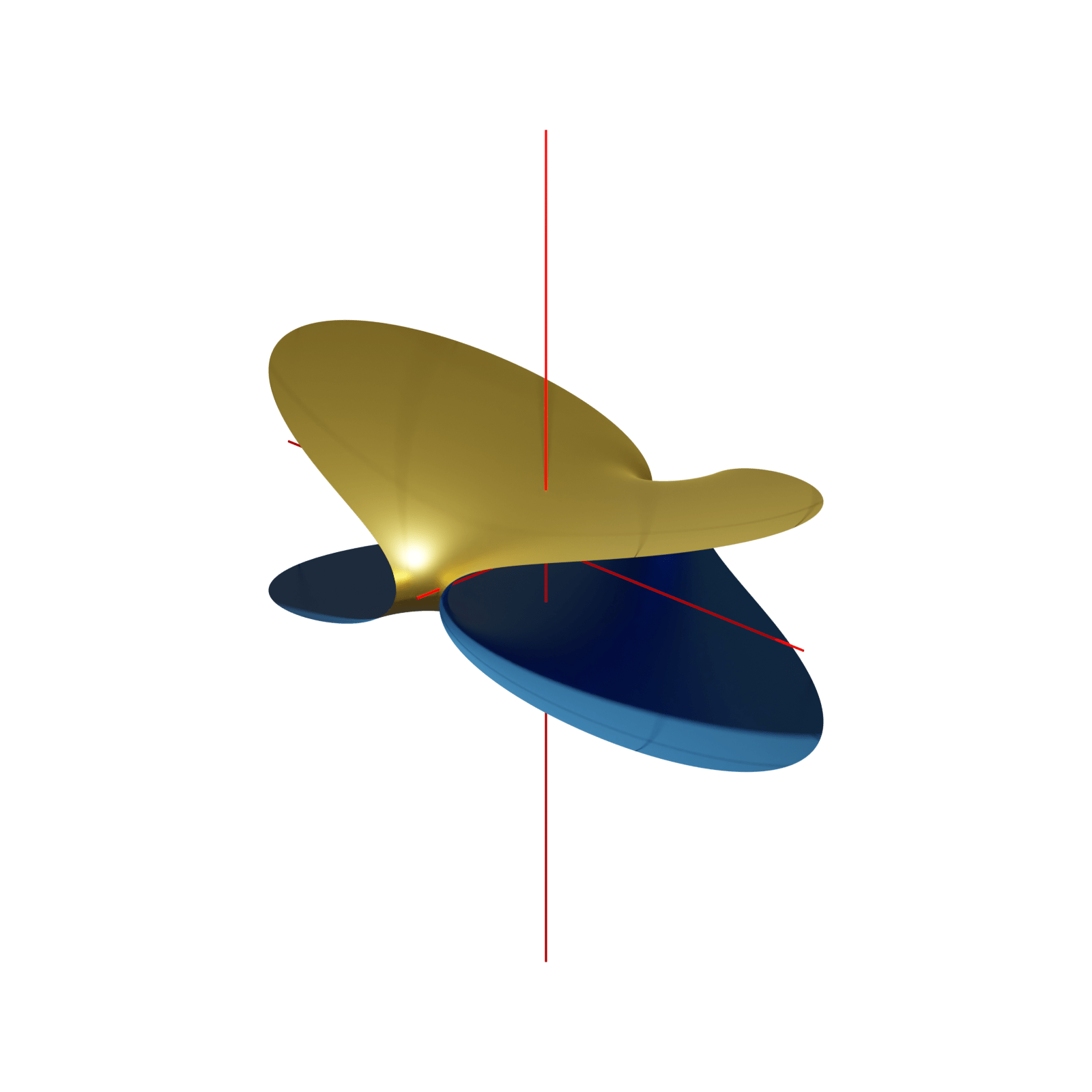}};
\begin{scope}[axis]
\draw[->](\apar,0,0)--(5,0,0)node[below right={1ex and 0},inner sep=0]{$x_1$};
\draw[->](0,\bpar,0)--(0,5,0)node[below]{$x_2$};
\end{scope}
\pgfresetboundingbox
\path(5,0,0)(0,5,0);
\end{ellipsoid}~

\bigskip\bigskip

\begin{ellipsoid}[apar=2,bpar=4,cpar=5]
\draw(0,0,0)node[scale=\unitscale]{\includegraphics[width={12cm}]{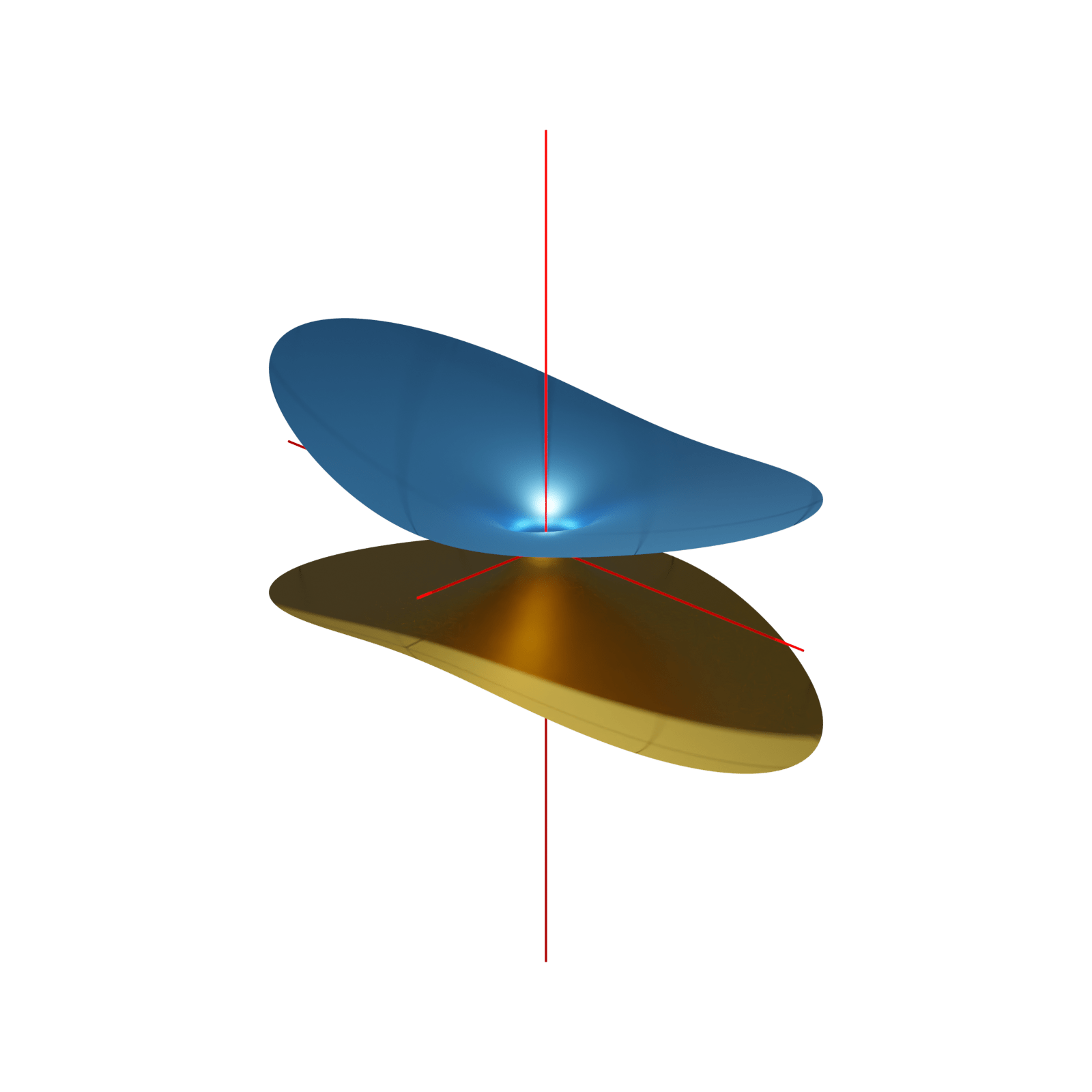}};
\begin{scope}[axis]
\draw[->](\apar,0,0)--(5,0,0)node[below right={1ex and 0},inner sep=0]{$x_1$};
\draw[->](0,\bpar,0)--(0,5,0)node[below]{$x_2$};
\end{scope}
\pgfresetboundingbox
\path(5,0,0)(0,5,0);
\end{ellipsoid}
\hfill
\begin{ellipsoid}[apar=2,bpar=4,cpar=5]
\draw(0,0,0)node[scale=\unitscale]{\includegraphics[width={12cm}]{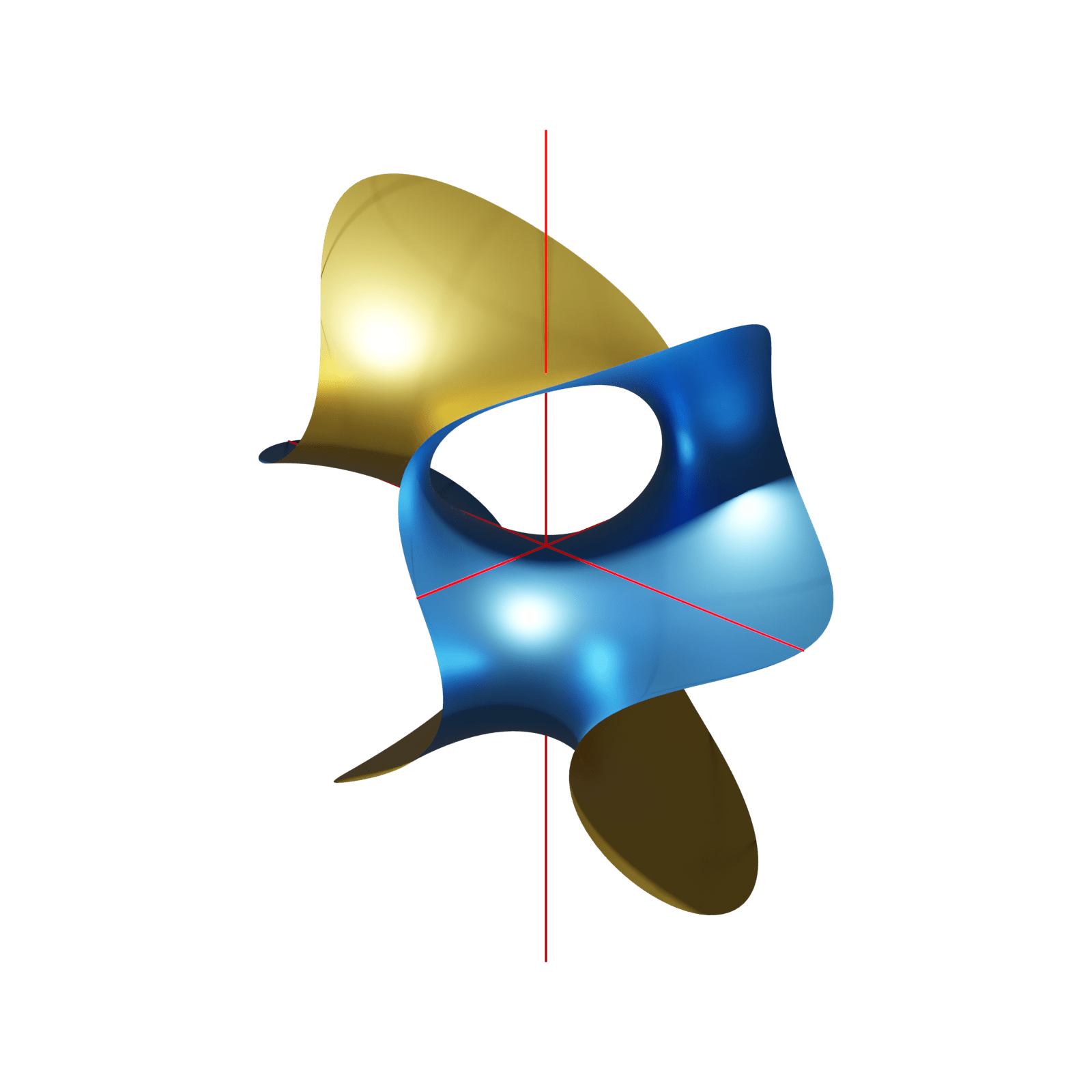}};
\begin{scope}[axis]
\draw[->](\apar,0,0)--(5,0,0)node[below right={1ex and 0},inner sep=0]{$x_1$};
\draw[->](0,\bpar,0)--(0,5,0)node[below]{$x_2$};
\end{scope}
\pgfresetboundingbox
\path(5,0,0)(0,5,0);
\end{ellipsoid}~
\caption{Three $\dih_2$-equivariant free boundary minimal annuli in $M_\apar$ for $\apar=(2,4,5)$ and 
a conjectural, $\dih_2$-equivariant solution with genus one and connected boundary.}%
\label{fig:annulus}%
\vspace*{-2.75ex}
\end{figure}

Property \ref{thm:annulus-ii} ensures that the three solutions $\solA_1,\solA_2,\solA_3$ are indeed distinct, and in fact noncongruent if the parameters $\apar_1,\apar_2,\apar_3$ are all different, because there is no ambient isometry interchanging segments $\xi_\iota$ of different length. 
(See Lemma~\ref{lem:D2annulus} about the general structure of $\dih_2$-equivariant annuli in $M_\apar$.) 
Simulations of the free boundary minimal annuli from Theorem~\ref{thm:annulus} are visualised in the first three images of Figure~\ref{fig:annulus}. 
In analogy with the existence result \cite[Theorem~1.1 for $g=1$]{CarlottoFranzSchulz2022} in $\B^3$, we conjecture that every ellipsoid also contains an embedded, $\dih_2$-equivariant free boundary minimal surface with genus one and connected boundary 
(cf.~Remark~\ref{rem:sweepout-g1b1} and Figure \ref{fig:annulus}, last image). 


When employing min-max methods to construct solutions with nontrivial topology, it is essential to prove that the limit of the resulting min-max sequence is not just a topological disc. 
In \cite{FranzSchulz2023,CarlottoFranzSchulz2022}, this step relies on Nitsche's \cite{Nitsche1985} uniqueness result for free boundary minimal discs in $\B^3$.  
However, it is evident that this argument does not generalise to ellipsoids. 
In fact, Theorem~\ref{thm:solU} stated below provides a counterexample even in the class of $\dih_2$-equivariant solutions, reaffirming the dramatic nonuniquenss of free boundary minimal discs in $M_\apar$.  
Instead, our proof of Theorem~\ref{thm:annulus} must rely on a much weaker property of $\dih_2$-equivariant discs in $M_\apar$ (cf.~Corollary~\ref{cor:D2disc}). 

\begin{theorem}\label{thm:solU}
For any $\alpha\in\interval{0,\infty}^3$ satisfying $\apar_3<\apar_1\apar_2/(\apar_1+\apar_2)$ the ellipsoid $M_\apar$ contains a $\dih_2$-equivariant free boundary minimal disc $\solU$ with with the following properties.  
\begin{enumerate}[label={\normalfont(\roman*)},nosep]
\item\label{thm:solU-i}   $\solU$ is nonplanar with area $\hsd^2(\solU)<\hsd^2(\solD_3)$.  
\item\label{thm:solU-ii}  $\solU$ contains the segments $\xi_1\cup\xi_2$. 
\item\label{thm:solU-iii} $\solU$ is equivariantly stable. 
\end{enumerate} 
\end{theorem}

\begin{proof} 
Let $\xi_1^+\vcentcolon=M_\apar\cap\{x_1\geq0=x_2=x_3\}$ and $\xi_2^+\vcentcolon=M_\apar\cap\{x_2\geq0=x_1=x_3\}$. 
Then $\xi_1^+\cup\xi_2^+$ is a Jordan curve in $\R^3$ with endpoints on $\partial M_\apar$. 
By \cite[§\,4.6, Theorem~2]{DHKW1992} there exists an area-minimising disc $\Sigma$ spanning $\xi_1^+\cup\xi_2^+$ with partially free boundary on $\partial M_\apar$. 
The assumption on $\apar$ implies that the competitor $\Gamma=(\solD_1\cup\solD_2)\cap\{x_1,x_2,x_3\geq0\}$ has area 
\(\hsd^2(\Gamma)=\tfrac{1}{4}\pi(\apar_2\apar_3+\apar_1\apar_3)
<\tfrac{1}{4}\pi\apar_1\apar_2=\tfrac{1}{4}\hsd^2(\solD_3)\). 
Being minimising, $\hsd^2(\Sigma)\leq\hsd^2(\Gamma)$. 
By the Schwarz reflection principle, the surface 
\[
\solU=\Sigma\cup(\rotation_1\Sigma)\cup(\rotation_2\Sigma)\cup(\rotation_1\rotation_2\Sigma)
\] 
is a $\dih_2$-equivariant free boundary minimal disc with area 
$\hsd^2(\solU)=4\hsd^2(\Sigma)<\hsd^2(\solD_3)$ containing $\xi_1\cup\xi_2$. 
In particular, $\solU$ must be nonplanar (see Figure~\ref{fig:solU}). 
Being equivariantly area-minimising, $\solU$ is equivariantly stable. 
\end{proof} 

\begin{figure}%
\centering
\begin{ellipsoid}[apar=6,bpar=4,cpar=2]
\draw(0,0,0)node[scale=\unitscale]{\includegraphics[width={12cm}]{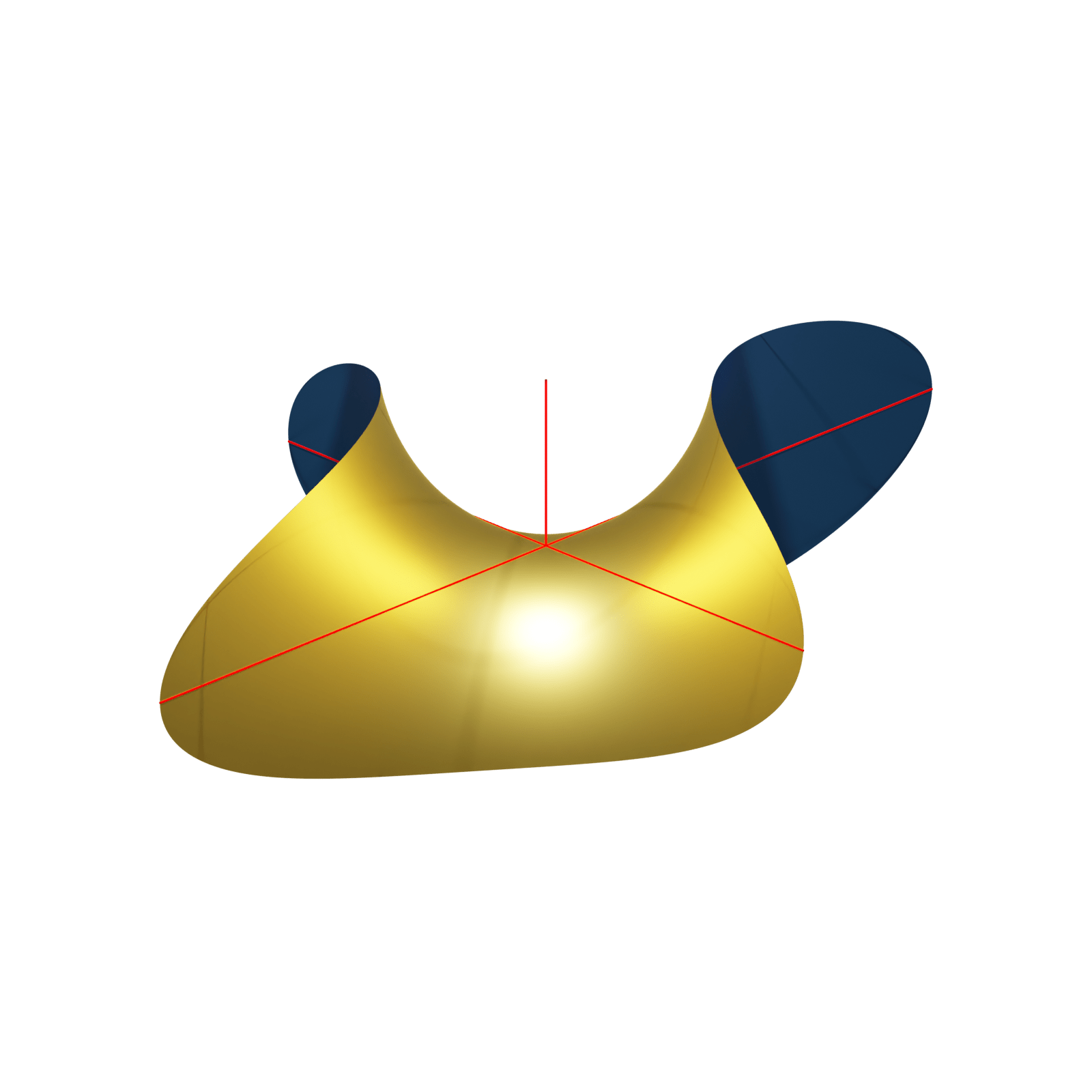}};
\pgfresetboundingbox
\begin{scope}[axis]
\draw[->](\apar,0,0)--(7,0,0)node[anchor=30,inner sep=1pt]{$x_1$};
\draw[->](0,\bpar,0)--(0,7,0)node[anchor=180-30,inner sep=1pt]{$x_2$};
\path[red!90!black](0,0,0)--(\apar,0,0)node[near end,below]{$\xi_1$};
\path[red!90!black](0,0,0)--(0,\bpar,0)node[near end,below left,inner sep=1pt]{$\xi_2$};
\end{scope}
\end{ellipsoid}
\caption{A nonplanar, $\dih_2$-equivariant free boundary minimal disc $\solU$ in $M_\apar$ for $\apar=(6,4,2)$.  }%
\label{fig:solU}%
\end{figure}

\begin{remark}
Lima and Menezes proved that free boundary minimal surfaces in the Euclidean unit ball $\B^3$ satisfy the 
\emph{two-piece property} \cite[Theorem~2]{LimaMenezes2021}, which states that the equatorial disc divides any compact, embedded free boundary minimal surface of $\B^3$ in exactly two connected components. 
A key step in their proof is based on \cite[Lemma~2]{LimaMenezes2021} stating that 
any immersed, stable, partially free boundary minimal surface in $\B^3$ with fixed boundary contained in the equatorial disc is necessarily totally geodesic.
This property does not generalise to ellipsoids and the surface constructed in Theorem \ref{thm:solU} is a counterexample. 
\end{remark}

\begin{remark}
The respective condition on $\apar$ in Theorems \ref{thm:solC}, \ref{thm:solS} and \ref{thm:solU} is not expected to be sharp but it ensures relevant area estimates. 
It is an interesting problem to determine the set of all $\apar\in\interval{0,\infty}^3$ for which the respective existence result remain true 
(cf.~\cite[Remark~2.11]{HaslhoferKetover}). 
\end{remark}
 
\paragraph{Acknowledgements.} 
The author would like to thank Giada Franz and Alessandro Carlotto for helpful comments and discussions. 
This project has received funding from the European Research Council (ERC) under the European Union's Horizon 2020 research and innovation programme (grant agreement No. 947923), 
and from the Deutsche Forschungsgemeinschaft (DFG, German Research Foundation) under Germany's Excellence Strategy EXC 2044 -- 390685587, Mathematics M\"unster: Dynamics--Geometry--Structure, and the Collaborative Research Centre CRC 1442, Geometry: Deformations and Rigidity.

\section{Sweepout construction}\label{sec:sweepout}

Given any $\alpha\in\interval{0,\infty}^3$ let $M_\apar$ be as in \eqref{eqn:ellipsoid} and $\solD_3=M_\apar\cap\{x_3=0\}$ as in \eqref{eqn:planar}. 
The notion of equivariant sweepout is defined e.\,g.~in \cite[Definition~1.1]{FranzSchulz2023}.

\begin{lemma}\label{lem:sweepout-C}
There exists a $\dih_1$-sweepout $\{\Sigma_t^{\solC}\}_{t\in[0,1]}$ of $M_{\apar}$ such that  
\begin{itemize}[nosep]
\item $\hsd^2(\Sigma_0^{\solC})=\hsd^2(\Sigma_1^{\solC})=0$ 
and $\hsd^2(\Sigma_t^{\solC})<2\hsd^2(\solD_3)$ for all $t\in\interval{0,1}$,
\item $\Sigma_t^{\solC}$ is a topological disc intersecting the $x_1$-axis orthogonally for all $t\in\interval{0,1}$.
\end{itemize}
\end{lemma}

\begin{proof}
We follow part (I) in the proof of \cite[Theorem 5.1]{FranzSchulz2023} and employ comparable notation.  
Let $B_\varepsilon(p_1)$ denote the open ball of radius $\varepsilon>0$ around the point $p_1=(\apar_1,0,0)\in\partial M_\apar$. 
Let $D_\varepsilon\vcentcolon=\solD_3\setminus B_\varepsilon(p_1)$
and $D_{t,\varepsilon}\vcentcolon=\sqrt{1-t^2}D_\varepsilon+(0,0,\apar_3t)$ for any $t\in\interval{-1,1}$. 
Then $D_{t,\varepsilon}\subset M_\apar$ and if $\varepsilon>0$ is sufficiently small, $D_{t,\varepsilon}$ is a topological disc. 
With $0<t_0<1$ and $\varepsilon_0>0$ to be chosen, let $\varepsilon\colon\Interval{t_0,1}\to\intervaL{0,\varepsilon_0}$ be a continuous function of $t$ such that $\varepsilon(t)\to0$ as $t\to1$ and define  
\begin{align}\label{eqn:20240611}
\Omega_{t}&\vcentcolon=\bigcup_{\tau\in[-t,t]}D_{\tau ,\varepsilon(t)}, & 
\Sigma_{t}&\vcentcolon=\overline{\partial\Omega_{t}\setminus\partial M_\apar}
\end{align}
as in \cite[(25)]{FranzSchulz2023}
for all $t\in\Interval{t_0,1}$, where $\partial\Omega_{t}$ refers to the topological boundary of the set $\Omega_{t}\subset\R^3$. 
Then $\Sigma_t$ is the union of $D_{\pm t,\varepsilon(t)}$ with a connecting ribbon. 
In particular, $\Sigma_t$ is $\dih_1$-equivariant.
By the coarea formula, there exists a constant $c_\apar$ depending only on $\apar$ such that the area of the ribbon is bounded from above by $c_\apar\varepsilon_0 t$ and we have 
\begin{align}
\label{eqn:area}
\hsd^2(\Sigma_t)&\leq 2(1-t^2)\hsd^2(\solD_3)+c_\apar\varepsilon_0 t
<2\hsd^2(\solD_3)
\end{align}
for all $t\in\Interval{t_0,1}$
by choosing $\varepsilon_0=(t_0/c_\apar)\hsd^2(\solD_3)$. 
As $t$ decreases from $t_0$ to $0$ we intend to deform $\Sigma_t$ continuously and $\dih_1$-equivariantly into a neighbourhood of the point $-p_1\in\partial M_\apar$ without violating the area bound \eqref{eqn:area} and such that $\hsd^2(\Sigma_t)\to0$ as $t\searrow0$.   
This is made possible by the so-called catenoid estimate \cite{KetoverMarquesNeves2020}*{Proposition~2.1 and Theorem~2.4} if $t_0>0$ is chosen sufficiently small (see also \cite[§\,2.3]{HaslhoferKetover}).  
The estimates can be carried out explicitly as detailed in 
\cite[§\,2]{CarlottoFranzSchulz2022} and \cite[§\,5]{FranzSchulz2023}. 
\end{proof}

\begin{lemma}\label{lem:sweepout-annulus}
For each $\iota\in\{1,2,3\}$ there exists a $\dih_2$-sweepout $\{\Sigma_t^{\iota}\}_{t\in[0,1]}$ of $M_\apar$ such that 
\begin{itemize}[nosep]
\item $\hsd^2(\Sigma_0^{\iota})=\hsd^2(\Sigma_1^{\iota})=0$ and $\hsd^2(\Sigma_t^{\iota})<2\hsd^2(\solD_3)$ for all $t\in\interval{0,1}$,
\item $\Sigma_t^{\iota}$ is an annulus which is disjoint from the $x_\iota$-axis for all $t\in\interval{0,1}$ and intersects the other two axes orthogonally.  
\end{itemize}
\end{lemma}

\begin{proof}
For $\iota=2$ it suffices to replace $D_\varepsilon$ with $\solD_3\setminus(B_\varepsilon(p_1)\cup B_\varepsilon(-p_1))$ 
in the proof of Lemma~\ref{lem:sweepout-C}. 
For $\iota=1$ we additionally replace the point $p_1$ by $p_2=(0,\apar_2,0)$ and then follow exactly the same approach. 
With these modifications, $\Sigma_t$ defined as in \eqref{eqn:20240611} becomes $\dih_2$-equivariant and
resembles the union of the horizontal discs $D_{\pm t,\varepsilon(t)}$ with two connecting ribbons. 
Estimate \eqref{eqn:area} still holds for all $t\in\Interval{t_0,1}$, possibly with a different constant $c_\apar$. 
As $t$ decreases from $t_0$ to $0$, the catenoid estimate allows us to $\dih_2$-equivariantly widen the two ribbons, deforming $\Sigma_t$ continuously into a neighbourhood of the segment $\xi_k$. 

For $\iota=3$, we return to the proof of Lemma~\ref{lem:sweepout-C} and replace $D_\varepsilon$ with $\solD_3\setminus B_\varepsilon(0)$. 
Then $\Sigma_t$ defined as in \eqref{eqn:20240611} is the $\dih_2$-equivariant union of the planar annuli $D_{\pm t,\varepsilon(t)}$ with a connecting central tube (instead of connecting ribbons with boundary). 
As $t$ decreases from $t_0$ to $0$ the catenoid estimate allows us to $\dih_2$-equivariantly widen the tube without violating the area estimate \eqref{eqn:area}, deforming $\Sigma_t$ continuously into a neighbourhood of the equator $\partial\solD_3$.   
\end{proof}

\begin{lemma}\label{lem:sweepout-S}
There exists a $\dih_1$-sweepout $\{\Sigma_t^{\solS}\}_{t\in[0,1]}$ of $M_{\apar}$ such that  
\begin{itemize}[nosep]
\item $\hsd^2(\Sigma_0^{\solS})=\hsd^2(\Sigma_1^{\solS})=\hsd^2(\solD_3)$ 
and $\hsd^2(\Sigma_t^{\solS})<3\hsd^2(\solD_3)$ for all $t\in\interval{0,1}$, 
\item $\Sigma_t^{\solS}$ is a topological disc containing the segment $\xi_1$ for all $t\in\interval{0,1}$.
\end{itemize}
\end{lemma}

\begin{proof}
We closely follow the proof of \cite[Lemma~2.2]{CarlottoFranzSchulz2022} but for genus $g=0$ rather than $g\geq1$. 
Given $\varepsilon>0$ and $p^{\pm}=(0,\pm\apar_2,0)$ consider the sets 
\begin{align*}
D_\varepsilon^{\pm}&\vcentcolon=\solD_3\setminus B_\varepsilon(p^{\pm}), &
D_\varepsilon^0&\vcentcolon=D_\varepsilon^{+}\cap D_\varepsilon^{-}, &
D_{t,\varepsilon}^{\pm}&\vcentcolon=\sqrt{1-t^2}D_\varepsilon^{\pm}\pm(0,0,\apar_3t) 
\end{align*}
for any $t\in\Interval{0,1}$.
With $0<t_0<1$ and $\varepsilon_0>0$ to be chosen, let $\varepsilon\colon\Interval{t_0,1}\to\intervaL{0,\varepsilon_0}$ be a continuous function of $t$ such that $\varepsilon(t)\to0$ as $t\to1$ and define  
\begin{align*}
\Omega_t^{\pm}&\vcentcolon=\bigcup_{\tau\in[0,t]}D_{\tau,\varepsilon(t)}^{\pm}, &
S_t^{\pm}&\vcentcolon=\overline{\partial\Omega_t^{\pm}\setminus(\partial\B^3\cup D_0^{\pm})}, & 
\Sigma_t&\vcentcolon=S_t^{+}\cup D_{\varepsilon(t)}^0\cup S_t^{-}
\end{align*}
as in \cite[Eqn.~(4)]{CarlottoFranzSchulz2022}. 
Then $\Sigma_t$ is the union of the three parallel topological discs $D_{t,\varepsilon(t)}^{\pm}$, $D_{\varepsilon(t)}^0$ with two ribbons joining them. 
As in the proof of Lemma~\ref{lem:sweepout-C} there is a constant $c_\apar$ such that the total area of the two ribbons is bounded from above by $c_\apar\varepsilon_0t$ and we obtain 
\begin{align}
\label{eqn:area2}
\hsd^2(\Sigma_t)&\leq  (3-2t^2)\hsd^2(\solD_3)+c_\apar\varepsilon_0 t
<3\hsd^2(\solD_3)
\end{align}
for all $t\in\Interval{t_0,1}$
by choosing $\varepsilon_0=(t_0/c_\apar)\hsd^2(\solD_3)$. 
If $t_0>0$ is chosen sufficiently small, the catenoid estimate allows us to increase the width of the ribbons as $t$ decreases from $t_0$ to $0$ without violating the area bound \eqref{eqn:area2} such that $\Sigma_t$ is deformed smoothly and $\dih_1$-equivariantly into $\solD_3$ as $t\searrow0$. 
The corresponding estimates are detailed in \cite[§\,2]{CarlottoFranzSchulz2022}.  
\end{proof}

\begin{remark}\label{rem:sweepout-g1b1}
By setting $p^{\pm}\vcentcolon=\bigl(\apar_1\cos(\pm\tfrac{\pi}{4}),\apar_2\sin(\pm\tfrac{\pi}{4}),0\bigr)$
and $D_\varepsilon^{\pm}\vcentcolon=\solD_3\setminus(B_\varepsilon(p^{\pm})\cup B_\varepsilon(-p^{\pm}))$  
in the proof of Lemma~\ref{lem:sweepout-C} we can follow the same approach to construct a $\dih_2$-sweepout $\{\Sigma_t\}_{t\in[0,1]}$ of $M_\apar$ such that 
\begin{itemize}[nosep]
\item $\hsd^2(\Sigma_0)=\hsd^2(\Sigma_1)=\hsd^2(\solD_3)$ 
and $\hsd^2(\Sigma_t)<3\hsd^2(\solD_3)$ for all $t\in\interval{0,1}$,
\item $\Sigma_t$ contains $\xi_1\cup\xi_2$, and has genus one and connected boundary for all $t\in\interval{0,1}$.
\end{itemize}
\end{remark}

\section{Width estimate}\label{sec:width}

We recall (e.\,g.~from \cite[Definition~1.3]{FranzSchulz2023}) that the equivariant saturation $\Pi$ of a given equivariant sweepout $\{\Sigma_t\}_{t\in[0,1]}$ of $M_\apar$ is defined as the set of all $\{f(t,\Sigma_t)\}_{t\in[0,1]}$, where $f\colon[0,1]\times M_\apar\to M_\apar$ is smooth such that $f(t,\cdot)$ is an equivariant diffeomorphism for all $t\in[0,1]$ which coincides with the identity if $t\in\{0,1\}$. 
The corresponding min-max width of $\Pi$ is  
\[
W_\Pi\vcentcolon=\adjustlimits\inf_{\{\Lambda_t\}\in \Pi~}\sup_{t\in[0,1]}\hsd^2({\Lambda_t}).
\]
The min-max approach (cf.~\cite[Theorem~1.4]{FranzSchulz2023} and references therein) requires the strict inequality $W_\Pi>\max\{\hsd^2(\Sigma_0),\hsd^2(\Sigma_1)\}$. 
This estimate is typically proven by levering the (relative) isoperimetric inequality in the ambient space. 
We recall the notation $\solD_3\vcentcolon=M_\apar\cap\{x_3=0\}$.

\begin{lemma}[{Isoperimetric inequality in ellipsoids \cite[Corollary~1]{Ros2005}}]
\label{lem:isoperim}
Let $M_\apar$ be as in \eqref{eqn:ellipsoid} such that $\apar_3\geq\max\{\apar_1,\apar_2\}$. 
Then any finite perimeter set $F\subset M_\apar$ with Lebesgue measure $\hsd^3(F)=\frac{1}{2}\hsd^3(M_\apar)$ 
has relative perimeter 
\(P(F;M_\apar)\geq\hsd^2(\solD_3)\).
\end{lemma}

\begin{lemma}[Uniqueness of isoperimetric sets in ellipsoids]
\label{lem:isoperim-unique}
Let $M_\apar$ be as in \eqref{eqn:ellipsoid} such that $\apar_3>\max\{\apar_1,\apar_2\}$. 
Suppose $F\subset M_\apar$ has Lebesgue measure $\hsd^3(F)=\frac{1}{2}\hsd^3(M_\apar)$ 
and relative perimeter $P(F;M_\apar)\leq\hsd^2(\solD_3)$.
Then the relative boundary of $F$ in $M_\apar$ coincides with $\solD_3$. 
\end{lemma}

\begin{proof}
Towards a contradiction, assume that $\Sigma\neq \solD_3$ is the relative boundary of $F$ in $M_\apar$. 
The assumption $\hsd^3(F)=\frac{1}{2}\hsd^3(M_\apar)$ implies that $\Sigma$ intersects $\solD_3$. 
In particular, $x_3$ is not constant on $\Sigma$. 
Let $\lambda=\max\{\apar_1,\apar_2\}/\apar_3$ and 
$\tilde{F}\vcentcolon=\{(x_1,x_2,x_3)\in\R^3\mid (x_1,x_2,x_3/\lambda)\in F\}$. 
Then, $\hsd^3(\tilde F)=\lambda\hsd^3(F)=\frac{\lambda}{2}\hsd^3(M_\apar)
=\frac{1}{2}\hsd^3(M_{\tilde\apar})$, where $\tilde{\apar}\vcentcolon=(\apar_1,\apar_2,\lambda\apar_3)$. 
On the one hand, Lemma~\ref{lem:isoperim} implies 
$P(\tilde F;M_{\tilde\apar})\geq\hsd^2(M_{\tilde\apar}\cap\{x_3=0\})=\hsd^2(M_{\apar}\cap\{x_3=0\})$ 
because $\lambda\apar_3\geq\max\{\apar_1,\apar_2\}$. 
On the other hand, the area formula \cite[Theorem~8.1]{Maggi2012} implies  
$P(\tilde F;M_{\tilde\apar})<P(F;M_{\apar})\leq\hsd^2\bigl(M_\apar\cap\{x_3=0\}\bigr)$ 
because $x_3$ is nonconstant on $\Sigma$ and $\lambda<1$ by assumption. 
This contradiction proves the claim. 
\end{proof}

\begin{lemma}[Stability of the isoperimetric inequality in ellipsoids]
\label{lem:isoperim-stability}
Let $M_\apar$ be as in \eqref{eqn:ellipsoid} such that $\apar_3>\max\{\apar_1,\apar_2\}$. 
For every $\varepsilon>0$ there exists $\delta>0$ such that given $F\subset M_\apar$ 
with Lebesgue measure $\hsd^3(F)=\frac{1}{2}\hsd^3(M_\apar)$
and relative perimeter $P(F;Z)\leq\hsd^2(\solD_3)+\delta$, 
we have either $\hsd^3\bigl(F\symdiff(M_\apar\cap\{x_3\geq0\})\bigr)\leq\varepsilon$
or $\hsd^3\bigl(F\symdiff(M_\apar\cap\{x_3\leq0\})\bigr)\leq\varepsilon$. 
\end{lemma}

\begin{proof} 
We follow the proof of \cite[Lemma~3.6]{CarlottoFranzSchulz2022}. 
Towards a contradiction, suppose that there exist $\varepsilon>0$ and a sequence $\{F_k\}_{k\in\N}$ of finite perimeter sets satisfying 
$\hsd^3({F_k})=\frac{1}{2}\hsd^3(M_\apar)$ and 
\begin{align*}
\hsd^3\bigl(F_k\symdiff(M_\apar\cap\{\pm x_3\geq0\})\bigr)&\geq\varepsilon, &
P(F_k;M_\apar)&\leq\hsd^2(\solD_3)+\delta_k
\end{align*}
such that $\delta_k\to0$ as $k\to\infty$. 
The compactness result \cite[Theorem 12.26]{Maggi2012} for finite perimeter sets 
implies the existence of $F_\infty\subset M_\apar$ with finite perimeter such that a subsequence of $\{F_k\}_{k\in\N}$ satisfies $\hsd^3({F_\infty\symdiff F_k})\to 0$ as $k\to\infty$. 
In particular, $\hsd^3({F_\infty})=\frac{1}{2}\hsd^3(M_\apar)$. 
Moreover,  
\begin{align*}
P(F_\infty;M_\apar)\leq\liminf_{k\to\infty} P(F_k;M_\apar)=\hsd^2(\solD_3)
\end{align*}
since the perimeter is lower semicontinuous (cf.~\cite[Proposition 12.15]{Maggi2012}). 
Lemma~\ref{lem:isoperim-unique} then implies that 
either $\hsd^3\bigl(F_{\infty}\symdiff(M_\apar\cap\{x_3\geq0\})\bigr)=0$
or $\hsd^3\bigl(F_{\infty}\symdiff(M_\apar\cap\{x_3\leq0\})\bigr)=0$ in contradiction with our choice of the sequence $\{F_k\}_{k\in\N}$.
\end{proof}

\begin{lemma}[Width estimate]
\label{lem:width}
Let $M_\apar$ be as in \eqref{eqn:ellipsoid} such that $\apar_3>\max\{\apar_1,\apar_2\}$ and let $\{\Sigma_t\}_{t\in[0,1]}$ be any one of the equivariant sweepouts of $M_\apar$ constructed in Lemmata~\ref{lem:sweepout-C}--\ref{lem:sweepout-S}. 
Then the min-max width $W_\Pi$ of the corresponding equivariant saturation satisfies 
\[W_\Pi>\hsd^2(\solD_3)\geq\max\{\hsd^2(\Sigma_0),\hsd^2(\Sigma_1)\}.\]
\end{lemma}
 
\begin{proof}
Given the stability of the isoperimetric inequality in $M_\apar$ stated in Lemma~\ref{lem:isoperim-stability}, the proof is exactly the same as in \cite[Theorem~5.1~(II)]{FranzSchulz2023} (for the sweepouts from Lemmata~\ref{lem:sweepout-C} and \ref{lem:sweepout-annulus}) respectively \cite[Proposition~3.7]{CarlottoFranzSchulz2022} 
(for the sweepout from Lemma~\ref{lem:sweepout-S}).
\end{proof}

\section{Geometric and topological control}\label{sec:control}

In this section we prove Theorems~\ref{thm:solC}--\ref{thm:annulus}. 
We recall the fact that, since the ambient manifold $M_\apar$ is simply connected, any properly embedded surface in $M_\apar$ must be orientable. 
Moreover, as $M_\apar$ is strictly convex, \cite[Lemma 2.4]{FraserLi2014} implies that every free boundary minimal surfaces in $M_\apar$ is necessarily connected. 

Let $\{\Sigma_t\}_{t\in[0,1]}$ be any one of the equivariant sweepouts of $M_\apar$  constructed in Lemmata~\ref{lem:sweepout-C}--\ref{lem:sweepout-S}. 
Assuming $\apar_3>\apar_1,\apar_2$, the width estimate stated in Lemma~\ref{lem:width} and the mean-convexity of $\partial M_\apar$ ensure that the min-max theorem \cite[Theorem~1.4]{FranzSchulz2023} applies:  
In each case, we obtain a min-max sequence $\{\Sigma^j\}_{j\in\N}$ converging in the sense of varifolds to $m\Gamma$, where $\Gamma$ is a compact, connected, embedded, equivariant free boundary minimal surface in $M_\apar$ and where the multiplicity $m$ is a positive integer. 
Moreover, the width $W_\Pi$ coincides with $m\hsd^2(\Gamma)$. 
It remains to verify the desired properties of $\Gamma$. 
In particular, the topology of $\Gamma$ must be determined, because topology is not necessarily preserved under varifold convergence. 

We start with a classification of all planar free boundary minimal surfaces in $M_\apar$ which will allow us to conclude that the solutions constructed in Theorems \ref{thm:solC} and \ref{thm:solS} are in fact nonplanar.
 
\begin{lemma}\label{lem:planar}
Let $\Sigma$ be a free boundary minimal surface in $M_\apar$ which is planar in the sense that $\Sigma=M_\apar\cap P$ for some plane $P\subset\R^3$.  
Then $\Sigma\in\{\solD_1,\solD_2,\solD_3\}$ up to ambient isometries. 
\end{lemma}

\begin{proof} 
Since $M_\apar$ is strictly convex, \cite[Lemma 2.4]{FraserLi2014} implies that 
$\Sigma$ intersects $\solD_1$. 
If $\Sigma=\solD_1$ the proof concludes. 
Otherwise, the intersection $\xi=\Sigma\cap\solD_1$ is a straight line segment 
since both $\Sigma$ and $\solD_1$ are planar. 
The free boundary property of $\Sigma$ implies that $\xi$ meets $\partial\solD_1$ orthogonally in both its endpoints $\gamma(s_0)$ and $\gamma(t_0)$, where we have parametrised $\partial\solD_1$ by $\gamma(s)=(0,\apar_2\cos s,\apar_3\sin s)$. 
It is elementary to show that $(s_0,t_0)$ is necessarily a critical point of the function $f(s,t)=\abs{\gamma(s)-\gamma(t)}^2$. 
By determining all critical points of $f$ explicitly, we conclude that $\xi$ coincides with either  the major or the minor axis of the ellipse $\solD_1$, respectively with a diameter of $\solD_1$ in the case $\apar_2=\apar_3$. 
Up to a coordinate rotation in the $x_2$-$x_3$-plane we have $\xi=\solD_1\cap\solD_2$.
Repeating the argument with $\solD_3$ in place of $\solD_1$ we obtain necessarily 
$\Sigma\cap\solD_3=\solD_3\cap\solD_2$. 
In particular, $\Sigma$ contains both axis of the ellipse $\solD_2$. 
Since $\Sigma$ is planar, $\Sigma=\solD_2$ follows. 
\end{proof}  

The next results pertain the general structure of properly embedded, equivariant discs which are not necessarily free boundary minimal surfaces in $M_\apar$. 
We recall that $\xi_1=M_\apar\cap\{x_2=0=x_3\}$. 

\begin{lemma}\label{lem:Hurwitz}
Let $M_\apar$ be as in \eqref{eqn:ellipsoid} and let $\Sigma\subset M_\apar$ be any smooth, properly embedded, $\dih_1$-equivariant topological disc. 
Then either $\xi_1\subset\Sigma$ or $\Sigma\cap\xi_1$ contains exactly one point and the intersection is orthogonal. 
\end{lemma}

\begin{proof}
We recall that the group $\dih_1$ is generated by the rotation $\rotation_1$ of angle $\pi$ around $\xi_1$. 
Let $p\in\Sigma\setminus\xi_1$. 
Then $p\neq\rotation_1p\in\Sigma$. 
Since $\Sigma$ is connected, a curve $\gamma\subset\Sigma$ connects $p$ and $\rotation_1p$. 
If $\gamma$ is disjoint from $\xi_1$ then $\gamma\cup\rotation_1\gamma$ is a simple closed curve winding around $\xi_1$. 
This curve is contractible in $\Sigma$ because $\Sigma$ is a topological disc. 
Therefore, $\Sigma\cap\xi_1$ must be nonempty.  
 
By \cite[Lemma~3.4~(2)]{Ketover2016Equivariant} we have either $\xi_1\subset\Sigma$ (which would complete the proof) or $\Sigma\cap\xi_1$ is finite and every intersection is orthogonal. 
Let $1\leq j\in\N$ be the number of points in $\Sigma\cap\xi_1$.  
The quotient $\Sigma'=\Sigma/\dih_1$ is a connected topological surface with boundary and therefore has Euler-Characteristic $\chi(\Sigma')\leq1$. 
A variant of the Riemann--Hurwitz formula (see e.\,g.~\cite[§\,IV.3]{Freitag2011}) 
and the fact that $\dih_1\simeq\Z_2$ implies
$1=\chi(\Sigma)=2\chi(\Sigma')-j(2-1)$ 
and thus $j=2\chi(\Sigma')-1\leq1$. 
Therefore $j=1$ as claimed. 
\end{proof}

\begin{corollary}\label{cor:D2disc}
Any smooth, properly embedded, $\dih_2$-equivariant topological disc $\Sigma\subset M_\apar$ contains two of the three segments $\xi_1,\xi_2,\xi_3$. 
\end{corollary}

\begin{proof}
We recall that $\dih_2$ contains the rotation $\rotation_\iota$ of angle $\pi$ around $\xi_\iota$ for any $\iota\in\{1,2,3\}$. 
Hence, Lemma~\ref{lem:Hurwitz} implies that either $\xi_\iota\subset\Sigma$ or $\Sigma\cap\xi_\iota$ contains exactly one point. 
Being smooth and embedded, $\Sigma$ cannot contain $\xi_\iota$ for all $\iota\in\{1,2,3\}$. 
Therefore there exists $k\in\{1,2,3\}$ such that $p\in\xi_k\cap\Sigma$ is unique. 
Since $\Sigma$ is $\dih_2$-equivariant, $\rotation_\iota p\in\xi_k\cap\Sigma$ for every $\iota\in\{1,2,3\}$. 
Hence, $p=(0,0,0)$. 
Lemma~\ref{lem:Hurwitz} then implies that $\Sigma$ 
contains $\xi_\iota$ for all $\iota\in\{1,2,3\}\setminus\{k\}$ because $\Sigma$ is orthogonal to $\xi_k$ and thus tangent to $\xi_\iota$ at the origin. 
\end{proof}
 
\begin{lemma}\label{lem:equal-volumes}
Any smooth, compact, connected, properly embedded, $\dih_1$-equivariant surface $\Sigma\subset M_\apar$ which contains the segment $\xi_1$ divides $M_\apar$ into two equal volumes. 
\end{lemma}

\begin{proof}
The surface $\Sigma$ is connected by assumption and orientable because $M_\apar$ is simply connected. 
Thus, $M_\apar\setminus\Sigma$ has exactly two connected components $F_1,F_2\subset M_\apar$. 
Let $\nu_0$ be a unit normal vector for $\Sigma$ at $0\in\xi_1\subset\Sigma$ and recall that $\rotation_1$ denotes the generator of $\dih_1$. 
Then, $\rotation_1\nu_0=-\nu_0$ and therefore $\rotation_1F_1=F_2$. 
Since $\rotation_1$ is an isometry, the claim $\hsd^3(F_1)=\hsd^3(F_2)$ follows. 
\end{proof}

\begin{proof}[Proof of Theorem \ref{thm:solC}]
Let $m\solC$ be the limit of the min-max sequence resulting from applying \cite[Theorem~1.4]{FranzSchulz2023} to the $\dih_1$-sweepout $\{\Sigma_t^{\solC}\}_{t\in[0,1]}$ constructed in Lemma~\ref{lem:sweepout-C}. 
Every surface along the min-max sequence is a topological disc, hence the lower semicontinuity result \cite[Theorem~1.8]{FranzSchulz2023} implies that the first Betti number of $\solC$ vanishes. 
Since $M_\apar\subset\R^3$ does not contain any closed minimal surfaces, we directly obtain that $\solC$ has genus zero and connected boundary (cf.~\cite[Proposition~A.1]{FranzSchulz2023}), proving that $\solC$ is a free boundary minimal disc. 
\begin{enumerate}[label={\normalfont(\roman*)},wide]
\item[\ref{thm:solC-i}]
We recall that $\hsd^2(\Sigma_t^{\solC})<2\hsd^2(\solD_3)$ for all $t\in[0,1]$. 
Thus, 
\begin{align}\label{eqn:area-C}
\hsd^2(\solD_3)<W_\Pi=m\hsd^2(\solC)<2\hsd^2(\solD_3). 
\end{align}
Since $\apar_3\geq2\apar_1$ by assumption, 
$\hsd^2(\solD_1)=\pi\apar_2\apar_3\geq2\pi\apar_2\apar_1=2\hsd^2(\solD_3)$. 
Therefore, \eqref{eqn:area-C} implies $\solC\neq\solD_1$ and $\solC\neq\solD_3$.   
The claim $\solC\neq\solD_2$ follows from the fact that $\xi_1\not\subset\solC$ which we will prove below to obtain statement \ref{thm:solC-ii}. 
Lemma~\ref{lem:planar} then implies that $\solC$ is nonplanar. 

\item[\ref{thm:solC-ii}]
By Lemma~\ref{lem:Hurwitz}, the surface $\solC$ either intersects the segment $\xi_1$ exactly once and the intersection is orthogonal, or $\xi_1\subset\solC$. 
By Lemma~\ref{lem:sweepout-C}, every surface along the min-max sequence intersects $\xi_1$ orthogonally. 
The $\dih_1$-equivariance then implies that if $\xi_1\subset\solC$, the multiplicity $m\in\N$ is even 
by \cite[Theorem~3.2.iv]{Ketover2016FBMS} (see also \cite[Theorem~1.3.f]{Ketover2016Equivariant}); 
in particular, $m\geq2$.  
Moreover, $\solC$ divides $M_\apar$ into two equal volumes by Lemma~\ref{lem:equal-volumes}. 
The isoperimetric inequality stated in Lemma~\ref{lem:isoperim} 
then yields $\hsd^2(\solC)\geq\hsd^2(\solD_3)$ which contradicts inequality \eqref{eqn:area-C} for $m\geq2$. 
 
\item[\ref{thm:solC-iii}] 
The min-max theorem \cite[Theorem~1.4]{FranzSchulz2023} (see also \cite[Theorem 1.10]{Franz2023}) states that the equivariant index is bounded from above by the number of parameters in the sweepout. 
Thus, the $\dih_1$-equivariant index of $\solC$ is at most $1$. 
Statement~\ref{thm:solC-ii} implies that the unit normal on $\solC$ is $\dih_1$-equivariant.  
Hence any constant function on $\solC$ is $\dih_1$-equivariant. 
Since the Jacobi quadratic form is negative on nonzero constant functions, the $\dih_1$-equivariant index of $\solC$ is equal to $1$. 
We refer to \cite[§\,8]{Franz2023} for more details.   
\qedhere
\end{enumerate}
\end{proof}

\begin{proof}[Proof of Theorem \ref{thm:solS}]
Let $m\solS$ be the limit of the min-max sequence resulting from applying \cite[Theorem~1.4]{FranzSchulz2023} to the $\dih_1$-sweepout $\{\Sigma_t^{\solS}\}_{t\in[0,1]}$ constructed in Lemma~\ref{lem:sweepout-S}. 
As in the proof of Theorem \ref{thm:solC} we may apply \cite[Theorem~1.8]{FranzSchulz2023} to prove that $\solS$ is a topological disc.  
By construction, every surface along the min-max sequence contains the segment $\xi_1$. 
Consequently, $\xi_1\subset\solS$ which proves Claim~\ref{thm:solS-ii} and implies that the multiplicity $m$ is odd (see \cite[§\,7.3]{Ketover2016FBMS}). 
Moreover, being $\dih_1$-equivariant and containing $\xi_1$, the surface $\solS$ divides $M_\apar$ into two equal volumes by Lemma~\ref{lem:equal-volumes}. 
Therefore, the isoperimetric inequality stated in Lemma~\ref{lem:isoperim} implies $\hsd^2(\solS)\geq\hsd^2(\solD_3)$. 
Since $\hsd^2(\Sigma_t^{\solS})<3\hsd^2(\solD_3)$ for all $t\in[0,1]$ we have  
\begin{align}\label{eqn:area-S}
\hsd^2(\solD_3)<W_\Pi=m\hsd^2(\solS)<3\hsd^2(\solD_3). 
\end{align}
In particular, $m<3$, and being odd, $m=1$. 
Since $\apar_3\geq3\apar_2$ by assumption, 
$\hsd^2(\solD_2)=\pi\apar_1\apar_3\geq3\pi\apar_1\apar_2=3\hsd^2(\solD_3)$.
Therefore, \eqref{eqn:area-C} implies 
$\solC\neq\solD_2$ and $\solC\neq\solD_3$. 
Moreover, $\solC\neq\solD_1$ because $\xi_1\subset\solC$. 
Hence $\solC$ is nonplanar by Lemma~\ref{lem:planar}. 
This completes the proof of \ref{thm:solS-i}. 

As in the proof of Theorem~\ref{thm:solC}~\ref{thm:solC-iii}, we obtain that the $\dih_1$-equivariant index of $\solS$ is at most~$1$. 
Since $\solS$ is nonplanar containing $\xi_1$, the function $u(x_1,x_2,x_3)=x_3$ restricts to a nonzero, $\dih_1$-equivariant negative direction for the Jacobi quadratic form on $\solS$ which proves that the $\dih_1$-equivariant index of $\solS$ is equal to $1$ (cf.~\cite[§\,8]{Franz2023}). 
\end{proof}

\begin{lemma}\label{lem:D2annulus}
Let $M_\apar$ be as in \eqref{eqn:ellipsoid} and let $\Sigma\subset M_\apar$ be any smooth, properly embedded, $\dih_2$-equivariant annulus. 
Then, $\Sigma$ is disjoint from exactly one of the three segments $\xi_1,\xi_2,\xi_3$ and it intersects each of the other two segments exactly twice and orthogonally. 
\end{lemma}

\begin{proof}
Let $\gamma_1$ and $\gamma_2$ denote the two boundary components of $\Sigma$. 
Let $W_1$ and $W_2$ be the two connected components of $\partial M_\apar\setminus\gamma_1$ labelled such that $\gamma_2\subset W_2$.  
Suppose there exist $p\in\gamma_1$ and $\rotation\in\dih_2\setminus\{\operatorname{id}\}$ such that $\rotation p=p$. 
Then $\rotation\gamma_1=\gamma_1$ and $\rotation\gamma_2\subset\rotation W_2=W_1$ would be a third boundary component of $\Sigma$. 
This contradiction proves that $\Sigma$ does not contain any of the three segments $\xi_1,\xi_2,\xi_3$, because the endpoints of $\xi_\iota$ being fixed by $\rotation_\iota\in\dih_2$ cannot be on $\partial\Sigma$. 

Given any $k\in\{1,2,3\}$, the set $\Sigma\cap\xi_k$ is finite (possibly empty) by \cite[Lemma~3.4~(2)]{Ketover2016Equivariant}, and every occurring intersection is orthogonal. 
Let $j_k\in\N\cup\{0\}$ be the cardinality of $\Sigma\cap\xi_k$. 
As in the proof of Lemma~\ref{lem:Hurwitz} we consider the connected topological surface $\Sigma/\rotation_k$ with Euler-Characteristic $\chi(\Sigma/\rotation_k)\leq1$ and apply the Riemann--Hurwitz formula \cite[§\,IV.3]{Freitag2011} to obtain
\begin{align}\label{eqn:20240610}
0=\chi(\Sigma)=2\chi(\Sigma/\rotation_k)-j_k
\end{align}
and thus $j_k\in\{0,2\}$.  
If $j_1=2$ then $\chi(\Sigma/\rotation_1)=1$ implying that 
$\Sigma/\rotation_1$ has connected boundary and thus $\rotation_1\gamma_1=\gamma_2$. 
If additionally $j_2=2$ then $\rotation_2\gamma_2=\gamma_1$ 
and $\rotation_3\gamma_1=\rotation_2\circ\rotation_1\gamma_1=\gamma_1$. 
Consequently, $\Sigma/\rotation_3$ has two boundary components, implying 
$\chi(\Sigma/\rotation_3)=0$ and $j_3=0$ by \eqref{eqn:20240610}. 
We conclude that $\Sigma$ is disjoint from at least one of the segments in question. 

Towards a contradiction, suppose that $\Sigma$ is disjoint from two of the segments -- without loss of generality $\xi_1$ and $\xi_2$. 
Then $j_1=0=j_2$ and (arguing as above) $\rotation_1\gamma_1=\gamma_1=\rotation_2\gamma_1$. 
Let $p\in\gamma_1$. 
As shown in the first paragraph of the proof, $\rotation_k p \neq p$ for any $k\in\{1,2,3\}$. 
Hence, the set $\gamma_1\setminus\{p,\rotation_1(p)\}$ has two connected components $c_1$ and $c_2$ 
satisfying $\rotation_1c_1=c_2$; in particular, they are of equal length. 
Arguing similarly for $\gamma_1\setminus\{p,\rotation_2p\}$ we obtain 
$\rotation_2 p =\rotation_1 p$ and thus the contradiction $\rotation_3 p =\rotation_2\circ\rotation_1 p  =p$. 
Therefore $\Sigma$ is disjoint from exactly one of the three segments and the proof concludes. 
\end{proof}

\begin{proof}[Proof of Theorem \ref{thm:annulus}]
Let $\apar\in\interval{0,\infty}^3$ be arbitrary. 
If $\apar_1=\apar_2=\apar_3$, then $M_\apar\subset\R^3$ is a round ball 
and the claim follows by rescaling and rotating the critical catenoid in $\B^3$ suitably. 
Otherwise, we may assume $\apar_3>\max\{\apar_1,\apar_2\}$ up to a change of coordinates, such that 
$\hsd^2(\solD_3)=\min_{\ell\in\{1,2,3\}}\hsd^2(\solD_\ell)$. 
In particular, the width estimate stated in Lemma \ref{lem:width} applies.  
\begin{enumerate}[label={\normalfont(\roman*)},wide]
\item[\ref{thm:annulus-i}]
Given $\iota\in\{1,2,3\}$, let $m\solA_\iota$ be the limit of the min-max sequence 
$\{\Sigma^{\iota,j}\}_{j\in\N}$
 resulting from applying \cite[Theorem~1.4]{FranzSchulz2023} to the $\dih_2$-sweepout $\{\Sigma_t^{\iota}\}_{t\in[0,1]}$ of $M_\apar$ constructed in Lemma~\ref{lem:sweepout-annulus}. 
The corresponding area estimates imply 
\begin{align}\label{eqn:area-annulus}
\hsd^2(\solD_3)<W_\Pi=m\hsd^2(\solA_\iota)<2\hsd^2(\solD_3). 
\end{align}
Since every surface along the min-max sequence is an annulus, the topological lower semicontinuity results \cite[Theorems~1.8--9]{FranzSchulz2023} imply that $\solA_\iota$ has genus zero and at most two boundary components. 
Towards a contradiction, suppose that $\solA_\iota$ is a topological disc. 
Since $\solA_\iota$ is $\dih_2$-equivariant, Corollary~\ref{cor:D2disc} implies that 
$\solA_\iota$ contains two of the three segments $\xi_1,\xi_2,\xi_3$. 
However, by Lemma~\ref{lem:sweepout-annulus}, every surface along the min-max sequence intersects two of the segments orthogonally. 
Therefore, the $\dih_2\cong\Z_2\times\Z_2$-equivariance implies that the multiplicity $m$ is even by \cite[Theorem~3.2.iv]{Ketover2016FBMS}.  
Applying Lemmata~\ref{lem:equal-volumes} and \ref{lem:isoperim} yields $\hsd^2(\solA_\iota)\geq\hsd^2(\solD_3)$ which contradicts \eqref{eqn:area-annulus} for $m\geq2$.  
Therefore, $\solA_\iota$ is an annulus as claimed, satisfying \ref{thm:annulus-i}. 

\item[\ref{thm:annulus-ii}] 
It remains to prove that $\solA_\iota$ is disjoint from the segment $\xi_\iota$ and intersects the other two segments orthogonally in order to distinguish the three free boundary minimal annuli in question. 
Towards a contradiction suppose that $\solA_\iota$ intersects $\xi_\iota$. 
By Lemma~\ref{lem:D2annulus} the $\dih_2$-equivariant annulus 
$\solA_\iota$ must be disjoint from $\xi_\ell$ for some $\ell\in\{1,2,3\}\setminus\{\iota\}$. 
Let $\varepsilon>0$ be sufficiently small, such that the $\varepsilon$-neighbourhood $U_\varepsilon\solA_\iota$ is still disjoint from $\xi_\ell$ and such that 
$U_\varepsilon\solA_\iota\cap\partial M_\apar$ has two connected components $N_1$ and $N_2$, one around each boundary component of $\solA_\iota$. 
By \cite[Theorem~4.11]{FranzSchulz2023} we may apply a topological surgery procedure to all surfaces $\Sigma^{\iota,j}$ in the min-max sequence with sufficiently large $j\in\N$ to obtain 
$\dih_2$-equivariant surfaces $\tilde\Sigma^{\iota,j}\subset U_\varepsilon\solA_\iota$
such that the sequence $\{\tilde\Sigma^{\iota,j}\}_{j}$ still converges to $m\solA_\iota$ in the sense of varifolds. 

Since $\Sigma^{\iota,j}$ is an annulus, \cite[Lemma~3.3 and~3.6]{FranzSchulz2023}  imply that at most one connected component of $\tilde\Sigma^{\iota,j}$ is an annulus and its remaining connected components are all topological discs or spheres. 
Indeed, employing \cite[Definition~1.6]{FranzSchulz2023}, we have genus complexity 
$\gsum(\tilde\Sigma^{\iota,j})=\gsum(\Sigma^{\iota,j})=0$ and boundary complexity 
$\bsum(\tilde\Sigma^{\iota,j})\leq\bsum(\Sigma^{\iota,j})=1$ 
(recalling that $1+\bsum$ counts the number of boundary components of a connected surface).  
Since the limit $\solA_\iota$ is also annulus, the topological lower semi-continuity stated in \cite[Theorem~4.11]{FranzSchulz2023} then implies that exactly one connected component 
$\hat\Sigma^{\iota,j}$ of $\tilde\Sigma^{\iota,j}$ has annular topology, because $1=\bsum(\solA_\iota)\leq\bsum(\tilde\Sigma^{\iota,j})\leq1$. 
In particular, $\hat\Sigma^{\iota,j}$ is $\dih_2$-equivariant and has two boundary components $\hat\gamma_1,\hat\gamma_2\subset N_1\cup N_2$. 
Since $\rotation_\iota N_1=N_2$ we have $\rotation_\iota\hat\gamma_1=\hat\gamma_2$. 
However, the boundary components $\gamma_1,\gamma_2$ of the original surface $\Sigma^{\iota,j}$, which is $\dih_2$-equivariantly isotopic to one of the surfaces $\Sigma_t^{\iota}$ in the sweepout constructed in Lemma~\ref{lem:sweepout-annulus}, satisfy 
$\rotation_\iota\gamma_k=\gamma_k$ for both $k\in\{1,2\}$. 
We recall that $\hat\Sigma^{\iota,j}$ is obtained from $\Sigma^{\iota,j}$ through surgery 
in the sense of \cite[Definition~3.1~(c)]{FranzSchulz2023}, i.\,e.~by 
\begin{enumerate}[label={(\arabic*)},nosep]
\item\label{surgery1} discarding connected components.
\item\label{surgery2} cutting away a neck,
\item\label{surgery3} cutting away a half-neck formed by a single boundary component,
\item\label{surgery4} cutting away a half-neck formed by two different boundary components,
\end{enumerate}
and note that any $\dih_2$-equivariant surgery procedure involving only operations of type 
\ref{surgery1}--\ref{surgery3} preserve the property that $\rotation_\iota\gamma_k=\gamma_k$ for both $k\in\{1,2\}$. 
The first occurrence of operation \ref{surgery4} however irreparably reduces the boundary complexity $\bsum(\tilde\Sigma^{\iota,j})$ from $1$ to $0$. 
This contradiction proves the claim.

\item[\ref{thm:annulus-iii}]
The computation of the equivariant index is analogous to the proof of Theorem~\ref{thm:solC}~\ref{thm:solC-iii}. 
\qedhere
\end{enumerate}
\end{proof}

\setlength{\parskip}{.75ex plus 1pt minus 2pt}
\bibliography{fbms-bibtex}

\printaddress

\end{document}